\documentclass[12pt,a4paper]{elsarticle}
\usepackage[utf8]{inputenc}
\usepackage{lineno,amsmath,amssymb,amsthm}
\usepackage{mathtools}
\usepackage{fullpage}
\usepackage{indentfirst}

\newtheorem{theorem}{Theorem}[section]
\newtheorem{lemma}[theorem]{Lemma}
\newtheorem{proposition}[theorem]{Proposition}
\newtheorem{corollary}[theorem]{Corollary}
\newtheorem{remark}[theorem]{Remark}
\newtheorem{example}[theorem]{Example}

\numberwithin{equation}{section}

\DeclareMathOperator{\HS}{H.S}

\makeatletter
\renewcommand{\@cite}[2]{\textnormal{[{#1\if@tempswa, #2\fi}]}}
\makeatother

\newcommand{\HSnorm}[1]{\left\lVert #1 \right\rVert_{\HS}}

\begin{document}
\begin{frontmatter}
\title{Hilbert-Schmidt Norms of Self-Commutator and Cross-Commutator on $N_\eta$-type Quotient Modules in Polydisc}
\author{Yuchao Xiang}
\address{College of Mathematical Science, Chongqing Univeristy of Technology}
\ead{xiangyc@stu.cqut.edu.cn}

\author{Anjian Xu\fnref{myfootnote}}
\address{College of Mathematical Science, Chongqing Univeristy of Technology}
\ead{xuaj@cqut.edu.cn}
\fntext[myfootnote]{Corresponding Author}
\date{}

\begin{abstract}
    Let $\eta_1,\cdots,\eta_d$ be nonconstant one-variable inner functions satisfying $\eta_i(0)=0$ for all $1\leq i\leq d$. Denote by $M_\eta$ the submodule of the Hardy space $H^{2}(\mathbb{D}^{d})$ generated by the function family $\{\eta_i(z_i)-\eta_{i+1}(z_{i+1})\mid 1\leq i\leq d-1\}$. For the corresponding quotient module $N_\eta=H^2(\mathbb{D}^d)\ominus M_\eta$, we define the compressed multiplication operators $S_{z_i}=P_{N_\eta}M_{z_i}\big|_{N_\eta}$, where $P_{N_\eta}$ stands for the orthogonal projection from $H^2(\mathbb{D}^d)$ onto $N_\eta$. In this paper, we characterize the Hilbert–Schmidt property of the self-commutators $[S_{z_i}^*,S_{z_i}]$ and cross-commutators $[S_{z_i}^*,S_{z_j}]$ with $i\neq j$. Moreover, explicit formulas for their Hilbert–Schmidt norms are established whenever these commutators are Hilbert–Schmidt.
\end{abstract}


\end{frontmatter}


\section{Introduction}
Let $\mathbb{D}^d$ be the open unit polydisk in $\mathbb{C}^{d}$ and $H^2(\mathbb{D}^d)$ the Hardy space over the polydisk. A closed subspace $M\subset H^2(\mathbb{D}^d)$ invariant under every coordinate multiplication $M_{z_j}$ is called a \emph{submodule}. Its orthogonal complement $N=H^2(\mathbb{D}^d)\ominus M$ is called the \emph{quotient module} associated to $N$, the compressed shift $S_{z_j}$ is the compression of $M_{z_j}$ on $M$, i.e., $S_{z_j}=P_N M_{z_j}|_N$. An quotient module is called essentially reductive if $[S_{z_i},S_{z_j}^{\ast}]$ is compact for any $1\leq i,j\leq d$. Characterizations of submodules, reductivity of quotient modules, properties of $S_{z_j}$ and their interplay are important themes in multivariable operator theory.

In one variable, the famous Beurling's theorem \cite{Beurling} says that every submodule of \(H^2(\mathbb{D})\) was given by $\theta H^2(\mathbb{D})$ for some inner $\theta$. But for $d\ge 2$, there is no such characterization \cite{Rudin}, and it motivates the study of tractable submodules and quotient modules. Among them is the $N_{\phi}$-type module, introduced by Izuchi and Yang~\cite{IzuchiYang08,IzuchiNakaziSeto04}:
$N_{\phi}=H^{2}(\mathbb{D}^2)\ominus[z-\phi(w)]$ for an one-variable inner function $\phi$. They proved that $N_{\phi}$ is essentially reductive if and only if $\phi$ is a finite Blaschke product \cite{IzuchiYang08} (see also \cite{Douglas06}). Another interesting example was studied by Clark \cite{Clark} and Wang \cite{Wang}. Let $\eta_i(z_i)$ be nonconstant one-variable inner functions with $\eta_i(0)=0$, and set
\[M_{\eta}=[\eta_{1}(z_{1})-\eta_{2}(z_{2}),\dots,\eta_{d-1}(z_{d-1})-\eta_{d}(z_{d})],\qquad N_{\eta}=H^2(\mathbb{D}^d)\ominus M_{\eta}.\]
Clark proved that $N_{\eta}$ is essentially normal when every $\eta_j$ is an Blaschke product with finite order $d_{j}$ for $1\leq j\leq d$ and Wang proved the converse. Wang and Zhao \cite{WangZhao} later characterized essentially normal homogeneous quotient modules over the polydisk. The essential normality of submodules and quotient modules is one of central themes in multivariate operator theory, originating in Arveson's program \cite{Arveson02,Arveson05} and developed extensively for the polydisk by Douglas, Misra, Guo, Wang, and so on (to just name a few, one can see \cite{ChenGuo,Douglas06,DouglasMisra08,GuoWang} for further information). Based on these results, one naturally asks further that whether  $[S_{z_i},S_{z_j}^{\ast}]$ is in Schatten $p$ class, what is the Schatten $p$ norm if it is, which provides more refined information than compactness. Zou, Yu and Yang \cite{ZouYuYang} computed the Hilbert-Schmidt norms of the self-commutators and cross-commutators on $N_{\psi,\phi}=H^2(\mathbb{D}^2)\ominus[\psi(z)-\phi(w)]$ over the bidisk, with $\psi,\phi$ nonconstant inner functions and $\psi(0)=0$.

In this paper we extend the above type theorem \cite{ZouYuYang} to $N_{\eta}$-type quotient modules for arbitrary $d\ge2$. Our main result (Theorem \ref{thm:main}) is as follows.
\begin{theorem}
	If $D_i=\prod_{j\ne i}d_j<\infty$, where $d_j=\dim H^{2}(\mathbb{D}^2)\ominus \eta_j H^{2}(\mathbb{D}^d)$, then
	\[\HSnorm{[S_{z_i}^*,S_{z_i}]}^2
	=D_i(d-1)^2\Bigl[\frac{\pi^2}{3}-2\sum_{k=1}^{d-1}\frac{1}{k^2}\Bigr]
	+D_i-2D_i(d-1)\,|\eta_i'(0)|^2,\]
	and $[S_{z_i}^*,S_{z_i}]$ fails to be Hilbert--Schmidt when $D_i=\infty$.
\end{theorem}
We also compute the Hilbert-Schmidt norms of the cross-commutators, it shows that they are independent of the inner functions. Moreover, it is proved that the self-commutators lie in $\mathcal{S}^p$ for every $p>1$ but not in $\mathcal{S}^1$, while the cross-commutators belong to all $\mathcal{S}^p$ for $p>\frac{1}{2}$ and are trace class. It is similar to Arveson's $p$-essential normality \cite{Arveson02,Arveson05}.

The paper is organized as follows. In section 2, an orthonormal basis of $N_\eta$ is constructed and the weighted shift $B_d$ is introduced. In section 3 the main theorem is proved via the layer decomposition and unitary model. Finally, the cross-commutators and their Schatten class membership is discusses in section 4.

\section{Preliminaries}
\subsection{An orthonormal basis of $N_\eta$}
Let \(M_0\) be the diagonal submodule
\[
M_0=[z_{1}-z_{2},z_{2}-z_{3},\ldots,z_{d-1}-z_{d}]
\]
in $H^2(\mathbb{D}^{d})$ and \(N_{0}\) be the associated quotient module $H^2(\mathbb{D}^d)\ominus M_0$. The following lemma gives an orthonormal basis of $N_0$.
\begin{lemma}\label{lem:N0-basis}
    \begin{equation}\label{eq:el-def}
    e_l^{(d)}(z)=\frac{1}{\sqrt{C_{d,l}}}\sum_{|\alpha|=l}z^{\alpha}, l\ge0
    \end{equation}
    is an orthonormal basis of $N_0$, where \(C_{d,l}=\binom{l+d-1}{d-1}\), $\alpha=(\alpha_1,\cdots,\alpha_d)\in\mathbb{N}^{d}$, $|\alpha|=\alpha_1+\cdots+\alpha_d=l$ and $z^{\alpha}=z_1^{\alpha_1}\cdots z_d^{\alpha_d}$.
\end{lemma}
\begin{proof}
   Firstly, we show $e_l^{(d)}\perp M_0$. For each generator $g_{i}=z_i-z_{i+1}$ and each monomial $z^\alpha$, we have
   \[\langle g_{i} z^\alpha, e_l^{(d)}\rangle
   = \langle z^{\alpha+e_i}-z^{\alpha+e_{i+1}}, e_l^{(d)}\rangle = 0,\]
   since the two terms cancel by symmetry of the symmetrized sum defining $e_l^{(d)}$, where $e_{i}=\overbrace{(0,\cdots,0,1,0\cdots,0)}^{i-th}$. Hence $e_l^{(d)}\perp M_0$.

   Secondly, if $f=\sum_\alpha a_\alpha z^\alpha\in N_0$. Then 
   \[\langle f,(z_i-z_{i+1})z^\beta\rangle=0\text{ for all }i,\beta\] 
   implies 
   \[a_{\beta+e_i}=a_{\beta+e_{i+1}}.\] 
   Thus all coefficients with the same total degree $l$ are equal, i.e., $a_\alpha=a_\beta$ if $|\alpha|=|\beta|$, we assume it is $c_l$ if $|\beta|=l$. Hence
   \[f=\sum_{l\ge0} c_l \sum_{|\alpha|=l}z^\alpha= \sum_{l\ge0} c_l\sqrt{C_{d,l}}\, e_l^{(d)},\]
   so $\{e_l^{(d)}\}$ spans $N_0$. Therefore it is an orthonormal basis.
\end{proof}

Let $\eta_i$ be nonconstant one-variable inner functions satisfying $\eta_i(0)=0$, $i=1,\dots,d$. Define the submodule
\[M_{\eta}=[\eta_1(z_1)-\eta_2(z_2),\eta_2(z_2)-\eta_3(z_3),\dots,\eta_{d-1}(z_{d-1})-\eta_d(z_d)]\]
and its associated quotient module $N_{\eta}=H^2(\mathbb{D}^d)\ominus M_{\eta}$. For each $i$, let

\begin{equation}\label{eq:Ki-def}
   K_i=H^2(z_i)\ominus\eta_i(z_i)H^2(z_i)
\end{equation}
   be the model space of $\eta_i$, and set $d_i=\dim K_i\in\mathbb{N}\cup\{\infty\}$. Since $\eta_i(0)=0$, the constant function $1$ belongs to $K_i$. Take an orthonormal basis $\{\theta_{s_i}^{(i)}(z_i)\}_{s_i=0}^{d_i-1}$ of $K_i$ with $\theta_{0}^{(i)}=1$; for $s_i\ge1$, replacing $\theta_{s_i}^{(i)}$ by $\theta_{s_i}^{(i)}-\theta_{s_i}^{(i)}(0)$ and normalizing ensures $\theta_{s_i}^{(i)}(0)=0$ for $0\leq s_{i}\leq d_{i}-1$.

\begin{theorem}\label{thm:basis}
\begin{equation}\label{eq:Esl-def}
    E_{s,l}= \left(\prod_{i=1}^d \theta_{s_i}^{(i)}(z_i)\right)
    \frac{1}{\sqrt{C_{d,l}}}  \sum_{i_1+\cdots+i_d=l}
    \eta_1(z_1)^{i_1}\cdots\eta_d(z_d)^{i_d},\qquad0\le s_i<d_i,\ l\ge0,
\end{equation}
is an orthonormal basis of $N_{\eta}$, where $C_{d,l}$ as in the lemma.
\end{theorem}
\begin{proof}
    By Wold decomposition, for each $1\le i\le d$ we have
    \[H^2(z_i)=\bigoplus_{m=0}^{\infty}\eta_i(z_i)^m K_i .\]
    For every multi-index $s=(s_1,\dots,s_d)$ with $0\le s_i<d_i$, set
    \[H_{s_i}^2(z_i)=\overline{\operatorname{span}}\left\{\theta_{s_i}^{(i)}(z_i)\eta_i(z_i)^m:\ m\ge 0\right\}.\]
    Then
    \[H^2(\mathbb D^d)=\bigoplus_s\left(\bigotimes_{i=1}^d H_{s_i}^2(z_i)\right).\]
    Fix such an index $s$. Define for $m_i\ge 0$
    \[\begin{array}{llll}
    	U_s:&H^2(\mathbb D^d)&\longrightarrow&\bigotimes_{i=1}^d H_{s_i}^2(z_i)\\
    	&z_1^{m_1}\cdots z_d^{m_d}&\mapsto&\prod_{i=1}^d\theta_{s_i}^{(i)}(z_i)\eta_i(z_i)^{m_i}.\
    \end{array}\]  
    Since $\{z_1^{m_1}\cdots z_d^{m_d}\}$ is an orthonormal basis of $H^2(\mathbb D^d)$ and
    \[\left\{\prod_{i=1}^d\theta_{s_i}^{(i)}(z_i)\eta_i(z_i)^{m_i}:\ m_i\ge 0\right\}\]
    is an orthonormal basis of $\bigotimes_{i=1}^d H_{s_i}^2(z_i)$, the map $U_s$ can be extent to an unitary operator.

    We now check the image of $M_0$ under $U_s$. For $1\le i\le d-1$, $U_s$ intertwines multiplication by $z_i-z_{i+1}$ on $H^2(\mathbb D^d)$ with multiplication by $\eta_i(z_i)-\eta_{i+1}(z_{i+1})$ on the target space, up to the fixed factor $\prod_{j=1}^d\theta_{s_j}^{(j)}(z_j)$. Concretely, for any monomial $z_1^{m_1}\cdots z_d^{m_d}$,
    \[U_s\bigl((z_i-z_{i+1})z_1^{m_1}\cdots z_d^{m_d}\bigr)=(\eta_i(z_i)-\eta_{i+1}(z_{i+1}))\prod_{j=1}^d\theta_{s_j}^{(j)}(z_j)\eta_j(z_j)^{m_j}.\]
    This relation extends linearly to all polynomials: for any polynomial $p$,
   \begin{equation}\label{eq:intertwine-poly}
       U_s\bigl((z_i-z_{i+1})p(z)\bigr)
       =(\eta_i(z_i)-\eta_{i+1}(z_{i+1}))\,U_s(p).
    \end{equation}
    Since polynomials are dense in $H^2(\mathbb{D}^d)$ and $U_s$ is unitary (hence bounded), the relation~\eqref{eq:intertwine-poly} holds for all $f\in H^2(\mathbb{D}^d)$ with $(z_i-z_{i+1})f\in H^2(\mathbb{D}^d)$. Consequently, $U_s$ maps the closed ideal generated by $\{z_i-z_{i+1}\}$ onto the restriction to $\bigotimes_i H_{s_i}^2(z_i)$ of the closed ideal generated by $\{\eta_i-\eta_{i+1}\}$, i.e.,
    \[U_s(M_0)=M_\eta\cap\left(\bigotimes_{i=1}^d H_{s_i}^2(z_i)\right).\]
    Since $U_s$ is unitary, it preserves orthogonal complements, so
    \[U_s(N_0)=\left(\bigotimes_{i=1}^d H_{s_i}^2(z_i)\right)\ominus M_\eta.\]
    Applying $U_s$ to the orthonormal basis $\{e_l^{(d)}\}_{l\ge 0}$ of $N_0$ gives
    \[U_s(e_l^{(d)})=\left(\prod_{i=1}^d\theta_{s_i}^{(i)}(z_i)\right)\frac{1}{\sqrt{C_{d,l}}}\sum_{i_1+\cdots+i_d=l}\eta_1(z_1)^{i_1}\cdots\eta_d(z_d)^{i_d}.\]
    Since $U_s$ is unitary, these vectors form an orthonormal basis of
    $\bigl(\bigotimes_{i=1}^d H_{s_i}^2(z_i)\bigr)\ominus M_\eta$.

    Since the slices are mutually orthogonal, we may take the orthogonal direct sum over all admissible multi-indices $s$, we conclude that $\{E_{s,l}\}$ is an orthonormal basis of $N_\eta$.
    When $d=2$, we have $C_{2,l}=l+1$, and the basis reduces to
    \[\theta_{s_1}^{(1)}(z_1)\theta_{s_2}^{(2)}(z_2)\cdot\frac{1}{\sqrt{l+1}}\sum_{i_1+i_2=l}\eta_1(z_1)^{i_1}\eta_2(z_2)^{i_2}.\]
\end{proof}

\subsection{Weighted shifts}\label{subsec:shift}
By symmetry of $e_l^{(d)}$ in all variables, the compressions $P_{N_0}M_{z_i}|_{N_0}$ coincide for $1\le i\le d$. We denote this common operator by $B_d$.

\begin{lemma}\label{lem:Bd}
   $B_d$ is a weighted shift with respect to $\{e_l^{(d)}\}_{l\ge0}$, i.e.,
   \begin{equation}\label{eq:Bd-action}
      B_d\,e_l^{(d)}=\sqrt{\frac{l+1}{l+d}}\;e_{l+1}^{(d)},
   \end{equation}
   and its defect operators are diagonal
   \begin{equation}\label{eq:Bd-defect}
      (I-B_dB_d^*)\,e_l^{(d)}=\frac{d-1}{l+d-1}\,e_l^{(d)},\qquad (I-B_d^*B_d)\,e_l^{(d)}=\frac{d-1}{l+d}\,e_l^{(d)}.
   \end{equation}
\end{lemma}
\begin{proof}
    The coefficient $\langle M_{z_i}e_l^{(d)},e_{l+1}^{(d)}\rangle$ is independent of $i$ by symmetry. A direct computation gives
    \[\langle z_i e_l^{(d)},\,e_{l+1}^{(d)}\rangle=\frac{1}{\sqrt{C_{d,l}\,C_{d,l+1}}}\cdot C_{d,l}=\sqrt{\frac{C_{d,l}}{C_{d,l+1}}}=\sqrt{\frac{l+1}{l+d}},\]
    using $C_{d,l+1}=\frac{l+d}{l+1}C_{d,l}$. Since $z_i e_l^{(d)}$ is homogeneous of degree $l+1$, its projection of this vector onto $N_0$ is a constant multiple of $e_{l+1}^{(d)}$, giving \eqref{eq:Bd-action}.

    The adjoint acts by
    \[B_d^* e_l^{(d)}=\sqrt{\frac{l}{l+d-1}}\;e_{l-1}^{(d)}\enspace(l\ge1),\qquad B_d^* e_0^{(d)}=0.\]
    Hence
    \[B_d^*B_d e_l^{(d)}=\frac{l+1}{l+d}\,e_l^{(d)}\enspace(l\ge0),\qquad B_dB_d^* e_l^{(d)}=\frac{l}{l+d-1}\,e_l^{(d)}\enspace(l\ge1),\qquad B_dB_d^* e_0^{(d)}=0.\]
    The formula for $B_dB_d^*$ also holds at $l=0$. Therefore
    \[(I-B_dB_d^*)e_l^{(d)}=\Bigl(1-\frac{l}{l+d-1}\Bigr)e_l^{(d)}=\frac{d-1}{l+d-1}\,e_l^{(d)},\]
    \[(I-B_d^*B_d)e_l^{(d)}=\Bigl(1-\frac{l+1}{l+d}\Bigr)e_l^{(d)}=\frac{d-1}{l+d}\,e_l^{(d)}.\]
\end{proof}

For $1\le i\le d$, let $S_{\eta_i}=P_{N_\eta}M_{\eta_i}|_{N_\eta}$ be the compression of multiplication by $\eta_i(z_i)$. The following lemma shows that each $S_{\eta_i}$ is a weighted shift with respect to the orthonormal basis \(\{E_{s,l}\}\).

\begin{lemma}\label{lem:Seta}
	For every $1\le i\le d$,
	\begin{equation}\label{eq:Seta-action}
		S_{\eta_i} E_{s,l}=\sqrt{\frac{C_{d,l}}{C_{d,l+1}}}\, E_{s,l+1}=\sqrt{\frac{l+1}{l+d}}\, E_{s,l+1}.
	\end{equation}
\end{lemma}
\begin{proof}
	Multiplication by $\eta_i(z_i)$ sends $\theta^{(i)}_{s_i}\eta_i^m$ to $\theta^{(i)}_{s_i}\eta_i^{m+1}$, hence it maps each summand $H^2_{s_i}(z_i)=\overline{\operatorname{span}}\{\theta^{(i)}_{s_i}\eta_i^m:m\ge0\}$ into itself. Consequently $\eta_i E_{s,l}$ lies in the slice $\mathcal H_s=\bigotimes_{j=1}^d H^2_{s_j}(z_j)$, and
	\[\eta_i E_{s,l}=\left(\prod_{j=1}^d\theta^{(j)}_{s_j}(z_j)\right)\frac{1}{\sqrt{C_{d,l}}}\sum_{k_1+\cdots+k_d=l}\eta_1(z_1)^{k_1}\cdots\eta_i(z_i)^{k_i+1}\cdots\eta_d(z_d)^{k_d}.\]
	The slices $\{\mathcal H_s\}_s$ are mutually orthogonal, and
	\[N_\eta=\bigoplus_s(\mathcal H_s\ominus M_\eta),\qquad\mathcal H_s\cap N_\eta=\overline{\operatorname{span}}\{E_{s,a}:a\ge0\}.\]
	Thus $P_{N_\eta}(\eta_i E_{s,l})$ belongs to $\overline{\operatorname{span}}\{E_{s,a}:a\ge0\}$, so
	\[S_{\eta_i}E_{s,l}=\sum_{a\ge0}\langle \eta_i E_{s,l},\,E_{s,a}\rangle\,E_{s,a}.\]
	
	Since in $\mathcal H_s$, the family $\{\prod_j\theta^{(j)}_{s_j}\eta_j^{m_j}\}$ is orthonormal, each $E_{s,a}$ is a linear combination of $\prod_j\theta^{(j)}_{s_j}\eta_j^{m_j}$ with $\sum_j m_j=a$, hence is homogeneous of total $\eta$-degree $a$. Since $\eta_i E_{s,l}$ has $\eta$-degree $l+1$, so
	\[\langle \eta_i E_{s,l},\,E_{s,a}\rangle=0\qquad(a\neq l+1),\]
	and therefore
	\[S_{\eta_i}E_{s,l}=\langle \eta_i E_{s,l},\,E_{s,l+1}\rangle\,E_{s,l+1}.\]
	
	It remains to find the coefficient. Write
	\[E_{s,l+1}=\left(\prod_{j=1}^d\theta^{(j)}_{s_j}(z_j)\right)\frac{1}{\sqrt{C_{d,l+1}}}\sum_{p_1+\cdots+p_d=l+1}\eta_1(z_1)^{p_1}\cdots\eta_d(z_d)^{p_d}.\]
	Since $\prod_j\theta^{(j)}_{s_j}$ has norm $1$ and $\{\eta_1^{a_1}\cdots\eta_d^{a_d}\}$ is orthonormal,
	\[\langle \eta_i E_{s,l},\,E_{s,l+1}\rangle=\frac{1}{\sqrt{C_{d,l}\,C_{d,l+1}}}\sum_{\substack{k_1+\cdots+k_d=l\\ p_1+\cdots+p_d=l+1}}\bigl\langle\eta_i^{k_i+1}\prod_{j\neq i}\eta_j^{k_j},\;\eta_1^{p_1}\cdots\eta_d^{p_d}\bigr\rangle.\]
	The inner product is $1$ if \[\begin{cases}
		p_j = k_j + 1 & j=i\\
		p_j = k_j & j\neq i
	\end{cases},\] and $0$ otherwise. Thus each of the $C_{d,l}$ multi-indices $(k_1,\dots,k_d)$ summing to $l$ contributes $1$, yielding
	\[\langle \eta_i E_{s,l},\,E_{s,l+1}\rangle=\frac{C_{d,l}}{\sqrt{C_{d,l}\,C_{d,l+1}}}=\sqrt{\frac{C_{d,l}}{C_{d,l+1}}}=\sqrt{\frac{l+1}{l+d}}.\]
	So
	\[S_{\eta_i}E_{s,l}=\sqrt{\frac{C_{d,l}}{C_{d,l+1}}}\,E_{s,l+1}=\sqrt{\frac{l+1}{l+d}}\,E_{s,l+1}.\] 
	The expression is independent of $i$, so all $S_{\eta_i}$ coincide.
\end{proof}

\section{The Hilbert-Schmidt norm of self-commutators}
Fix $1\le i\le d$ and let $S_{z_i}=P_{N_\eta}M_{z_i}|_{N_\eta}$ be the compressed coordinate multipliers on $N_\eta$. Set
\[D_i=\prod_{j\neq i}d_j\in\mathbb N\cup\{\infty\}.\]

\subsection{Action of $S_{z_i}$ on the basis vectors}
We begin by computing the action of $S_{z_i}$ directly on the orthonormal basis $\{E_{s,l}\}$. This formula will serve as the foundation for the layer decomposition. Let $A_i:K_i\to K_i$ be the Jordan block of the model space $K_i=H^2(z_i)\ominus\eta_i H^2(z_i)$, that is, the compression $A_i f=P_{K_i}(z_i f)$. Define the rank-one operator
\[T_i:K_i\to K_i,\qquad T_i f=\langle f,\ \bar z_i\eta_i\rangle\,1 .\]
Since $\eta_i(0)=0$, the function $\eta_i$ factors as $\eta_i(z_i)=z_i g(z_i)$ with $g\in H^2(\mathbb{D})$, hence $\bar z_i\eta_i=g\in H^2(\mathbb{D})$. For any $h\in H^2(\mathbb{D})$,
\[\langle \bar z_i\eta_i,\,\eta_i h\rangle
=\langle g,\,\eta_i h\rangle
=\langle z_i g h,\,1\rangle=0,\]
so $\bar z_i\eta_i\in K_i$. Because $\eta_i$ is inner function, $\|\bar z_i\eta_i\|=\|\eta_i\|=1$. The defect operator identities
\begin{equation}\label{eq:defect-model}
I-A_i^*A_i=T_i^*T_i,\qquad I-A_iA_i^*=T_iT_i^*
\end{equation}
hold on $K_i$. These are standard for model spaces with $\eta_i(0)=0$: the backward shift $A_i^*$ satisfies $A_i^*1=0$, and $\operatorname{ran}T_i=\mathbb{C}\cdot1$, giving $A_i^*T_i=0$; a direct computation shows $\|T_i^*f\|^2=\|f\|^2-\|A_i^*f\|^2$ for all $f\in K_i$, which is equivalent to~\eqref{eq:defect-model}.

Given a multi-index $s=(s_1,\dots,s_d)$, let
\[\Theta_s(z)=\prod_{j=1}^d\theta_{s_j}^{(j)}(z_j),\qquad
e_l^{(d)}(\eta;z)=\frac{1}{\sqrt{C_{d,l}}}\sum_{|\alpha|=l}\eta_1(z_1)^{\alpha_1}\cdots\eta_d(z_d)^{\alpha_d},\]
and so $E_{s,l}=\Theta_s\cdot e_l^{(d)}(\eta;\cdot)$.

\begin{proposition}\label{prop:Sz-action}
   For every basis vector $E_{s,l}$,
   \begin{equation}\label{eq:Sz-basis}
      S_{z_i}E_{s,l}=(A_i\theta_{s_i}^{(i)})\prod_{j\neq i}\theta_{s_j}^{(j)}\,e_l^{(d)}(\eta)+\langle z_i\theta_{s_i}^{(i)},\,\eta_i\rangle\,\eta_i{\prod_{j\neq i}\theta_{s_j}^{(j)}\,e_l^{(d)}(\eta)}.
   \end{equation}
\end{proposition}
\begin{proof}
   With respect to the orthogonal decomposition $H^2(z_i)=K_i\oplus\eta_i H^2(z_i)$, we have
   \begin{equation}\label{eq:model-split}
       z_i\theta^{(i)}_{s_i}=A_i\theta^{(i)}_{s_i}+\langle z_i\theta^{(i)}_{s_i},\,\eta_i\rangle\,\eta_i .
   \end{equation}
   This decomposition is valid, because $A_i\theta_{s_i}^{(i)}\in K_i$ and the second term is exactly the projection of $z_i\theta_{s_i}^{(i)}$ onto $\eta_i H^2(z_i)$. Indeed, for $\theta\in K_i$, the function $\bar{\eta}_i z_i\theta$ is orthogonal to every monomial $z^n$ with $n\ge 1$ (as $\theta\perp z_i\eta_i H^2$), so its projection onto $H^2$ is the constant $\langle z_i\theta,\eta_i\rangle$.

   By ~\eqref{eq:model-split}, we have
   \[z_i E_{s,l}=(A_i\theta^{(i)}_{s_i})\prod_{j\neq i}\theta^{(j)}_{s_j}\,e_l^{(d)}(\eta)+\langle z_i\theta^{(i)}_{s_i},\,\eta_i\rangle\,\eta_i{\prod_{j\neq i}\theta_{s_j}^{(j)}\,e_l^{(d)}(\eta)}.\]
\end{proof}

Equation \eqref{eq:Sz-basis} shows that $S_{z_i}$ does not change the indices $s_j$ for $j\neq i$, it is a key property. The factors and the $\eta$-sum does not change except the $i$-th factor. For the $i$-th factor, it is either $A_i\theta_{s_i}^{(i)}$, which is an element of $K_i$, or a constant multiple of $\eta_i$. The constant multiple of $\eta_i$ is absorbed by $S_{\eta_i}$ later. This observation motivates the following layer decomposition.

\subsection{Layer decomposition and the unitary model}
For a fixed multi-index $s'=(s_j)_{j\neq i}$ with $0\le s_j<d_j$, set
\[N_{s'}:=\overline{\operatorname{span}\bigl\{E_{s,l}:\ 0\le s_i<d_i,\ l\ge 0\bigr\}},\]
where $s=(s_1,\dots,s_d)$ has its $j$-th entries ($j\neq i$) prescribed by $s'$. Hence each $N_{s'}$ \emph{reduces} $S_{z_i}$, and we have the orthogonal decomposition
\[N_\eta=\bigoplus_{s'} N_{s'},\qquad [S_{z_i}^*,S_{z_i}]=\bigoplus_{s'}\bigl[(S_{z_i}|_{N_{s'}})^*,\,S_{z_i}|_{N_{s'}}\bigr].\]
The Hilbert--Schmidt norm splits over the layers:
\begin{equation}\label{eq:layer-split}
    \HSnorm{[S_{z_i}^*,S_{z_i}]}^2     =\sum_{s'}\HSnorm{\bigl[(S_{z_i}|_{N_{s'}})^*,\,S_{z_i}|_{N_{s'}}\bigr]}^2 .
\end{equation}

We now build a unitary model for the restriction of $S_{z_i}$ to each layer $N_{s'}$. Recall the weighted shift $B_d$ on $N_0$ from Subsection \ref{subsec:shift}. Define
\[U_{s'}:N_{s'}\longrightarrow K_i\otimes N^{(d)}_0,\qquad U_{s'}E_{s,l}=\theta^{(i)}_{s_i}(z_i)\otimes e_l^{(d)} .\]
Both $\{E_{s,l}\}_{s_i,l}$ and $\{\theta^{(i)}_{s_i}\otimes e_l^{(d)}\}_{s_i,l}$ are orthonormal bases of $N_{s'}$ and $K_i\otimes N_0$, respectively, and $U_{s'}$ maps one onto the other; hence $U_{s'}$ is unitary.

By Lemma \ref{lem:Seta}, $S_{\eta_i}$ acts as $I\otimes B_d$ under this identification,
\begin{equation}\label{eq:Seta-model}
   U_{s'}\,S_{\eta_i}\big|_{N_{s'}}\,U_{s'}^*=I\otimes B_d .
\end{equation}
Indeed,
\[U_{s'}S_{\eta_i}E_{s,l}
=U_{s'}\Bigl(\sqrt{\frac{l+1}{l+d}}\,E_{s,l+1}\Bigr)
=\sqrt{\frac{l+1}{l+d}}\;\theta^{(i)}_{s_i}\otimes e_{l+1}^{(d)}
=(I\otimes B_d)\,U_{s'}E_{s,l}.\]

\begin{proposition}\label{prop:model}
   For each $s'$,
   \begin{equation}\label{eq:Sz-model}
      U_{s'}\,S_{z_i}\big|_{N_{s'}}\,U_{s'}^*=A_i\otimes I+T_i\otimes B_d.
\end{equation}
\end{proposition}
\begin{proof}
   Apply $U_{s'}$ to both sides of~\eqref{eq:Sz-basis}. For the first term, expanding $A_i\theta_{s_i}^{(i)}$ in the basis of $K_i$ as in the proof of Proposition~\ref{prop:Sz-action} and using the definition of $U_{s'}$,
   \[U_{s'}\bigl((A_i\theta_{s_i}^{(i)})\prod_{j\neq i}\theta_{s_j}^{(j)}\,e_l^{(d)}(\eta)\bigr)=A_i\theta_{s_i}^{(i)}\otimes e_l^{(d)}=(A_i\otimes I)\,U_{s'}E_{s,l}.\]
   For the second term, since $T_i\theta_{s_i}^{(i)}=\langle\theta_{s_i}^{(i)},\bar z_i\eta_i\rangle\,1=\langle z_i\theta_{s_i}^{(i)},\eta_i\rangle\,1$,
   \[U_{s'}\bigl(\langle z_i\theta_{s_i}^{(i)},\eta_i\rangle\,\eta_i{\prod_{j\neq i}\theta_{s_j}^{(j)}\,e_l^{(d)}(\eta)}\bigr)=\langle z_i\theta_{s_i}^{(i)},\eta_i\rangle\,\sqrt{\frac{l+1}{l+d}} (I\otimes e_{l+1}^{(d)})=(T_i\theta_{s_i}^{(i)}) (I\otimes B_d e_l^{(d)})\]
   Summing the two contributions gives~\eqref{eq:Sz-model}.
\end{proof}

\subsection{Characterization of the self-commutator}
\begin{proposition}\label{prop:commutator}
   For each $s'$,
   \begin{equation}\label{eq:comm-model}
       U_{s'}\,\bigl[S_{z_i}^*,S_{z_i}\bigr]\big|_{N_{s'}}\,U_{s'}^*=T_iT_i^*\otimes(I-B_dB_d^*)-T_i^*T_i\otimes(I-B_d^*B_d).
   \end{equation}
\end{proposition}
\begin{proof}
    Set $S=A_i\otimes I+T_i\otimes B_d$. On $K_i$ we have
    \begin{equation}\label{eq:cross-vanish}
       A_i^*T_i=0,\qquad T_iA_i^*=0,
    \end{equation}
    and consequently $T_i^*A_i=0$ and $A_iT_i^*=0$. Indeed, $\operatorname{ran}T_i=\mathbb C\cdot 1$ and the backward shift formula on a model space gives $A_i^*1=0$, so $A_i^*T_i=0$; likewise $T_iA_i^*f=\langle A_i^*f,\bar z_i\eta_i\rangle 1=\langle f,A_i\bar z_i\eta_i\rangle 1$ and $A_i(\bar z_i\eta_i)=P_{K_i}(\eta_i)=0$ because $\eta_i\in\eta_i H^2\perp K_i$.

    Expanding $S^*S$ and $SS^*$,
    \[\begin{aligned}
       S^*S &=(A_i^*\otimes I+T_i^*\otimes B_d^*)(A_i\otimes I+T_i\otimes B_d)\\
       &=A_i^*A_i\otimes I+A_i^*T_i\otimes B_d+T_i^*A_i\otimes B_d^*+T_i^*T_i\otimes B_d^*B_d,\\[4pt]
       SS^* &=(A_i\otimes I+T_i\otimes B_d)(A_i^*\otimes I+T_i^*\otimes B_d^*)\\
       &=A_iA_i^*\otimes I+A_iT_i^*\otimes B_d^*+T_iA_i^*\otimes B_d+T_iT_i^*\otimes B_dB_d^*.
    \end{aligned}\]
    By \eqref{eq:cross-vanish} the cross terms vanish, so
    \[S^*S=A_i^*A_i\otimes I+T_i^*T_i\otimes B_d^*B_d,\qquad SS^*=A_iA_i^*\otimes I+T_iT_i^*\otimes B_dB_d^* .\]

    Subtracting and using $A_i^*A_i-A_iA_i^*=T_iT_i^*-T_i^*T_i$ from \eqref{eq:defect-model} gives
    \[[S^*,S]=T_iT_i^*\otimes(I-B_dB_d^*)-T_i^*T_i\otimes(I-B_d^*B_d),\]
\end{proof}

\subsection{The main theorem}
\begin{theorem}\label{thm:main}
    Let $N_\eta=H^2(\mathbb D^d)\ominus M_\eta$ with $d\ge2$, where $\eta_1,\dots,\eta_d$ are nonconstant one-variable inner functions with $\eta_j(0)=0$, and fix $1\le i\le d$. If $D_i=\prod_{j\ne i}d_j<\infty$, then
    \begin{equation}\label{eq:main}
       \HSnorm{[S_{z_i}^*,S_{z_i}]}^2=D_i(d-1)^2\left[\frac{\pi^2}{3}-2\sum_{k=1}^{d-1}\frac{1}{k^2}\right]+D_i-2D_i(d-1)\,|\eta_i'(0)|^2 .
    \end{equation}
    If $D_i=\infty$, then $[S_{z_i}^*,S_{z_i}]$ is not Hilbert-Schmidt.
\end{theorem}
\begin{proof}
For a fix layer $s'$ and write $T=T_i$, $B=B_d$, $\theta=\theta^{(i)}_{s_i}$. Since $Tf=\langle f,\bar z_i\eta_i\rangle\,1$ and $T^*f=f(0)\,\bar z_i\eta_i$ with $\|\bar z_i\eta_i\|=1$,
\[
TT^*\theta=\theta(0)\cdot 1,
\qquad
T^*T\theta
=\langle z_i\theta,\,\eta_i\rangle\,\bar z_i\eta_i .
\]
By Proposition~\ref{prop:commutator} and Lemma~\ref{lem:Bd},
\[
U_{s'}[S_{z_i}^*,S_{z_i}]U_{s'}^*
(\theta\otimes e_l^{(d)})
=\frac{d-1}{l+d-1}\,\theta(0)\,(1\otimes e_l^{(d)})
-\frac{d-1}{l+d}\,\langle z_i\theta,\eta_i\rangle\,
(\bar z_i\eta_i\otimes e_l^{(d)}).
\]
The vectors $1\otimes e_l^{(d)}$ and $\bar z_i\eta_i\otimes e_l^{(d)}$ lie in $K_i\otimes\operatorname{span}\{e_l^{(d)}\}$ and are generally not orthogonal. Their inner product is
\[
\langle 1\otimes e_l^{(d)},\ \bar z_i\eta_i\otimes e_l^{(d)}\rangle
=\langle 1,\bar z_i\eta_i\rangle_{K_i}\,
\|e_l^{(d)}\|^2
=\overline{\eta_i'(0)},
\]
since $\eta_i(0)=0$ implies $\langle 1,\bar z_i\eta_i\rangle=(\bar z_i\eta_i)(0)=\overline{\eta_i'(0)}$. Expanding the squared norm,
\begin{align*}
\bigl\lVert[S_{z_i}^*,S_{z_i}]E_{s,l}\bigr\rVert^2
={}&\frac{(d-1)^2}{(l+d-1)^2}\,|\theta(0)|^2
+\frac{(d-1)^2}{(l+d)^2}\,\bigl|\langle z_i\theta,\eta_i\rangle\bigr|^2\\
&-\frac{2(d-1)^2}{(l+d-1)(l+d)}\,
\operatorname{Re}\Bigl(\theta(0)\,
\overline{\langle z_i\theta,\eta_i\rangle}\;\overline{\eta_i'(0)}\Bigr).
\end{align*}
Summing over $s_i=0,\dots,d_i-1$ and using that $\{\theta^{(i)}_{s_i}\}$ is an orthonormal basis of $K_i$ containing unit vectors $1,\bar z_i\eta_i$, Parseval's identity yields
\[
\sum_{s_i}|\theta^{(i)}_{s_i}(0)|^2
=\sum_{s_i}|\langle\theta^{(i)}_{s_i},1\rangle|^2=1,
\qquad
\sum_{s_i}|\langle z_i\theta^{(i)}_{s_i},\eta_i\rangle|^2
=\sum_{s_i}|\langle\theta^{(i)}_{s_i},\bar z_i\eta_i\rangle|^2=1,
\]
and for the cross term, we have
\[
\sum_{s_i}\theta^{(i)}_{s_i}(0)\,
\overline{\langle z_i\theta^{(i)}_{s_i},\eta_i\rangle}
=\sum_{s_i}
\langle\bar z_i\eta_i,\ \theta^{(i)}_{s_i}\rangle
\langle\theta^{(i)}_{s_i},\ 1\rangle
=\langle\bar z_i\eta_i,\,1\rangle
=\eta_i'(0).
\]
Consequently,
\begin{equation}\label{eq:si-sum}
\sum_{s_i=0}^{d_i-1}
\bigl\lVert[S_{z_i}^*,S_{z_i}]E_{s,l}\bigr\rVert^2
=\frac{(d-1)^2}{(l+d-1)^2}
+\frac{(d-1)^2}{(l+d)^2}
-\frac{2(d-1)^2\,|\eta_i'(0)|^2}{(l+d-1)(l+d)} .
\end{equation}
The right-hand side is independent of~$s'$.

Now sum over $l\ge0$, we have 
\[
\sum_{l=0}^{\infty}\frac{1}{(l+d-1)^2}
=\frac{\pi^2}{6}-\sum_{k=1}^{d-2}\frac{1}{k^2},
\qquad
\sum_{l=0}^{\infty}\frac{1}{(l+d)^2}
=\frac{\pi^2}{6}-\sum_{k=1}^{d-1}\frac{1}{k^2},
\]
\[
\sum_{l=0}^{\infty}\frac{1}{(l+d-1)(l+d)}
=\sum_{l=0}^{\infty}\Bigl(\frac{1}{l+d-1}-\frac{1}{l+d}\Bigr)
=\frac{1}{d-1}.
\]
Summing \eqref{eq:si-sum} over $l$ gives, for each layer $s'$,
\begin{align*}
\HSnorm{\bigl[(S_{z_i}|_{N_{s'}})^*,\,S_{z_i}|_{N_{s'}}\bigr]}^2
&=(d-1)^2\Bigl[\frac{\pi^2}{3}
-\sum_{k=1}^{d-2}\frac{1}{k^2}-\sum_{k=1}^{d-1}\frac{1}{k^2}\Bigr]
-\frac{2(d-1)^2|\eta_i'(0)|^2}{d-1}\\
&=(d-1)^2\Bigl[\frac{\pi^2}{3}-2\sum_{k=1}^{d-1}\frac{1}{k^2}\Bigr]
+1-2(d-1)\,|\eta_i'(0)|^2 ,
\end{align*}
by $\sum\limits_{k=1}^{d-2}k^{-2}=\sum\limits_{k=1}^{d-1}k^{-2}-(d-1)^{-2}$.

By the layer decomposition \eqref{eq:layer-split}, the Hilbert--Schmidt norm of $[S_{z_i}^*,S_{z_i}]$ is this common layer contribution multiplied by the number $D_i=\prod_{j\neq i}d_j$ of layers. If $D_i<\infty$, this yields \eqref{eq:main}. If $D_i=\infty$,  Write $a=l+d-1$, $b=l+d$, the summand in \eqref{eq:si-sum} is
\[
(d-1)^2\Bigl[\Bigl(\frac1a-\frac1b\Bigr)^2
+\frac{2(1-|\eta_i'(0)|^2)}{ab}\Bigr]>0,
\]
since $a\neq b$ and $|\eta_i'(0)|\le1$ by the Schwarz lemma. Hence the sum in \eqref{eq:layer-split} diverges, and $[S_{z_i}^*,S_{z_i}]$ is not Hilbert--Schmidt.
\end{proof}

\begin{remark}
For $d=2$, let $j$ be the index different from $i$. Then $D_i=d_j$, $(d-1)^2=1$, and $\sum\limits_{k=1}^{1}k^{-2}=1$, so \eqref{eq:main} becomes
\[
\HSnorm{[S_{z_i}^*,S_{z_i}]}^2
=\frac{\pi^2}{3}\,D_i-D_i-2D_i\,|\eta_i'(0)|^2,
\]
which is just the formula in \textnormal{\cite{ZouYuYang}} for the case of bidisk.
\end{remark}

\begin{example}[Example: $d=3$ with finite Blaschke products]\label{rem:example}
To anchor the abstract formula in a concrete case, take $d=3$ and let
\[
\eta_1(z)=z^2,\quad \eta_2(z)=z^3,\quad \eta_3(z)=z.
\]
All three are finite Blaschke products vanishing at $0$. Their model space dimensions are $d_1=2$, $d_2=3$, $d_3=1$. For $i=1$, we have
\[
D_1=d_2d_3=3,\qquad |\eta_1'(0)|=0\;\;(\text{since }\eta_1(z)=z^2).
\]
With $d=3$, compute $\sum_{k=1}^{2}k^{-2}=1+\frac14=\frac54$. Then formula~\eqref{eq:main} gives
\[
\HSnorm{[S_{z_1}^*,S_{z_1}]}^2
=3\cdot 2^2\left[\frac{\pi^2}{3}-\frac52\right]+3-2\cdot3\cdot2\cdot0
=12\left(\frac{\pi^2}{3}-\frac52\right)+3
=4\pi^2-30+3=4\pi^2-27.
\]
For $i=3$, we have $D_3=d_1d_2=6$, $|\eta_3'(0)|=1$ (since $\eta_3(z)=z$), giving
\[
\HSnorm{[S_{z_3}^*,S_{z_3}]}^2
=6\cdot4\left[\frac{\pi^2}{3}-\frac52\right]+6-2\cdot6\cdot2\cdot1
=24\left(\frac{\pi^2}{3}-\frac52\right)-18
=8\pi^2-60-18=8\pi^2-78.
\]
Note that the commutator for $i=3$ ($\eta_3(z)=z$) is smaller than for $i=1$ ($\eta_1(z)=z^2$), illustrating the phenomenon that $|\eta_i'(0)|$ closer to $1$ reduces the commutator norm. The coefficient of $|\eta_i'(0)|^2$ in \eqref{eq:main} is negative. Hence the commutator norm decreases when $|\eta_i'(0)|$ increases.
\end{example}

\section{Cross-commutators and further remarks}

\subsection{The cross-commutators}

We now study the cross-commutators $[S_{z_i}^*,S_{z_j}]$ with $i\neq j$. Unlike the self-commutator case, they involve two distinct compressed operators and require a refinement of the layer decomposition.

Fix $i\neq j$ and set
\[
D_{ij}=\prod_{k\neq i,j}d_k\in\mathbb N\cup\{\infty\}.
\]

\subsubsection{Bi-layer decomposition}
For a multi-index $s''=(s_k)_{k\neq i,j}$ with $0\le s_k<d_k$, define
\[
N_{s''}=
\overline{\operatorname{span}}
\bigl\{
E_{s,l}:\ 0\le s_i<d_i,\ 0\le s_j<d_j,\ l\ge0
\bigr\},
\]
where $s$ has its $k$-th entries ($k\neq i,j$) prescribed by $s''$. By the same reasoning as in Section~3.2, the explicit formula~\eqref{eq:Sz-basis} for the action of $S_{z_i}$ (and the analogous formula for $S_{z_j}$) shows that these operators do not alter the indices $s_k$ for $k\neq i,j$. Hence each $N_{s''}$ reduces both $S_{z_i}$ and $S_{z_j}$, and
\begin{equation}\label{eq:bi-layer-split}
\HSnorm{[S_{z_i}^*,S_{z_j}]}^2
=\sum_{s''}\HSnorm{\bigl[(S_{z_i}|_{N_{s''}})^*,\,S_{z_j}|_{N_{s''}}\bigr]}^2 .
\end{equation}

Define a unitary map
\[
V_{s''}:N_{s''}\longrightarrow K_i\otimes K_j\otimes N^{(d)}_0,
\qquad
V_{s''}E_{s,l}=\theta^{(i)}_{s_i}(z_i)\otimes\theta^{(j)}_{s_j}(z_j)\otimes e_l^{(d)} .
\]

\subsubsection{Unitary model on the bi-layer}
We repeat the argument of Proposition~\ref{prop:model} on $N_{s''}$. Decomposing $z_i\theta^{(i)}_{s_i}$ as in \eqref{eq:model-split} and noting that $M_{z_i}$ leaves $\theta^{(j)}_{s_j}$ ($j\neq i$) unchanged, we obtain
\[
S_{z_i}E_{s,l}
=(A_i\theta^{(i)}_{s_i})\prod_{k\neq i}\theta^{(k)}_{s_k}\,e_l^{(d)}(\eta)
+\langle z_i\theta^{(i)}_{s_i},\eta_i\rangle\,\eta_i{\prod_{k\neq i}\theta_{s_k}^{(j)}\,e_l^{(d)}(\eta)}.
\]
Applying $V_{s''}$ and using $S_{\eta_i}\cong I\otimes I\otimes B_d$ (Lemma~\ref{lem:Seta}) together with $T_i\theta^{(i)}_{s_i}=\langle z_i\theta^{(i)}_{s_i},\eta_i\rangle\,1$ yields
\begin{equation}\label{eq:bi-model}
\begin{aligned}
V_{s''}\,S_{z_i}\big|_{N_{s''}}\,V_{s''}^*
&=A_i\otimes I\otimes I+T_i\otimes I\otimes B_d,\\
V_{s''}\,S_{z_j}\big|_{N_{s''}}\,V_{s''}^*
&=I\otimes A_j\otimes I+I\otimes T_j\otimes B_d .
\end{aligned}
\end{equation}
Both formulas involve the \emph{same} weighted shift $B_d$ (Lemma~\ref{lem:Seta} guarantees $S_{\eta_i}=S_{\eta_j}$), while the Jordan block and rank-one parts act on their respective $K$ factors.

\begin{proposition}\label{prop:cross-comm}
For each $s''$,
\[
V_{s''}\,\bigl[S_{z_i}^*,S_{z_j}\bigr]\big|_{N_{s''}}\,V_{s''}^*
=T_i^*\otimes T_j\otimes(B_d^*B_d-B_dB_d^*).
\]
\end{proposition}

\begin{proof}
Abbreviate $S=S_{z_i}$, $T=S_{z_j}$ in the bi-layer model. Using \eqref{eq:bi-model},
\[
\begin{aligned}
S^*&=A_i^*\otimes I\otimes I+T_i^*\otimes I\otimes B_d^*,\\
T&=I\otimes A_j\otimes I+I\otimes T_j\otimes B_d .
\end{aligned}
\]
Expanding $[S^*,T]=S^*T-TS^*$,
\[
\begin{aligned}
S^*T&=(A_i^*\otimes I\otimes I+T_i^*\otimes I\otimes B_d^*)(I\otimes A_j\otimes I+I\otimes T_j\otimes B_d)\\
    &=A_i^*\otimes A_j\otimes I+A_i^*\otimes T_j\otimes B_d+T_i^*\otimes A_j\otimes B_d^*+T_i^*\otimes T_j\otimes B_d^*B_d,\\[4pt]
TS^*&=(I\otimes A_j\otimes I+I\otimes T_j\otimes B_d)(A_i^*\otimes I\otimes I+T_i^*\otimes I\otimes B_d^*)\\
    &=A_i^*\otimes A_j\otimes I+T_i^*\otimes A_j\otimes B_d^*+A_i^*\otimes T_j\otimes B_d+T_i^*\otimes T_j\otimes B_dB_d^*.
\end{aligned}
\]
Subtracting, the mixed terms cancel, leaving
\[
[S^*,T]=T_i^*\otimes T_j\otimes(B_d^*B_d-B_dB_d^*).
\]
\end{proof}

\subsubsection{ The Hilbert--Schmidt norm of cross commutators}
Taking the basis vector \(\theta_{s_i}^{(i)} \otimes \theta_{s_j}^{(j)} \otimes e_l \in K_i \otimes K_j \otimes N_0\), from the unitary equivalence formula for the cross-commutator, we obtain
\[
\begin{aligned}
&\quad V_{s''} [S_{z_i}^*, S_{z_j}] V_{s''}^*
\left( \theta_{s_i}^{(i)} \otimes \theta_{s_j}^{(j)} \otimes e_l \right) \\
&= \left( T_i^* \theta_{s_i}^{(i)} \right) \otimes \left( T_j \theta_{s_j}^{(j)} \right)
\otimes \left( (B_d^* B_d - B_d B_d^*) e_l \right).
\end{aligned}
\]
From Lemma~\ref{lem:Bd}, the defect operators give
\[
(B_d^*B_d-B_dB_d^*)\,e_l^{(d)}
=\Bigl(\frac{l+1}{l+d}-\frac{l}{l+d-1}\Bigr)e_l^{(d)}
=\frac{d-1}{(l+d)(l+d-1)}\,e_l^{(d)},
\qquad l\ge0.
\]
then
\[
\bigl\|B_d^*B_d-B_dB_d^*\bigr\|_{\HS}^2
=\sum_{l=0}^{\infty}\frac{(d-1)^2}{(l+d)^2(l+d-1)^2},
\]
and
\[
\begin{aligned}
\sum_{l=0}^{\infty}\frac{(d-1)^2}{(l+d)^2(l+d-1)^2}
&=(d-1)^2\sum_{l=0}^{\infty}
\Bigl(\frac{1}{(l+d-1)^2}+\frac{1}{(l+d)^2}-\frac{2}{(l+d-1)(l+d)}\Bigr)\\
&=(d-1)^2\Bigl[\frac{\pi^2}{3}
-\sum_{k=1}^{d-2}\frac{1}{k^2}-\sum_{k=1}^{d-1}\frac{1}{k^2}\Bigr]
-2(d-1),
\end{aligned}
\]

Summing over all $D_{ij}$ bi-layers via \eqref{eq:bi-layer-split} yields the following theorem.

\begin{theorem}\label{thm:cross}
Let $N_\eta$ be as in Theorem~\ref{thm:main} and fix $i\neq j$. If $D_{ij}=\prod_{k\neq i,j}d_k<\infty$, then
\begin{equation}\label{eq:cross-main}
\begin{aligned}
\HSnorm{[S_{z_i}^*,S_{z_j}]}^2
=D_{ij}(d-1)^2\Bigl[\frac{\pi^2}{3}
-\sum_{k=1}^{d-2}\frac{1}{k^2}
-\sum_{k=1}^{d-1}\frac{1}{k^2}\Bigr]
-2D_{ij}(d-1).
\end{aligned}
\end{equation}
If $D_{ij}=\infty$, then $[S_{z_i}^*,S_{z_j}]$ is not Hilbert--Schmidt.
\end{theorem}

\begin{remark}
Unlike the self-commutator, the cross-commutator norm is independent of the inner functions---it depends only on $d$, $D_{ij}$, and the defect dimensions. This reflects the fact that the rank-one operators $T_i^*$ and $T_j$ enter as a pure tensor product whose Hilbert-Schmidt norm carries no $\eta_i$-dependence.
\end{remark}

\begin{remark}
For $d=2$, the set $\{1,2\}\setminus\{i,j\}$ is empty, so $D_{ij}=1$. With $(d-1)^2=1$ and $\sum_{k=1}^{1}k^{-2}=1$, \eqref{eq:cross-main} reduces to
\[
\HSnorm{[S_{z_1}^*,S_{z_2}]}^2
=\frac{\pi^2}{3}-1-2=\frac{\pi^2}{3}-3,
\]
which is structurally parallel to the corresponding formula of \cite{ZouYuYang}.
\end{remark}

\subsection{Schatten class membership}

The explicit spectral decomposition obtained in the proof of Theorem~\ref{thm:main} allows us to sharpen the Hilbert--Schmidt conclusion.

\begin{corollary}\label{cor:schatten}
Assume $D_i<\infty$. Then for every $1\le i\le d$, the self-commutator $[S_{z_i}^*,S_{z_i}]$ belongs to the Schatten--von Neumann class $\mathcal{S}^p$ for all $p>1$. The cross commutator $[S_{z_i}^*,S_{z_j}]$ ($i\neq j$) when $D_{ij}<\infty$ belongs to the Schatten-von Neumann class $\mathcal{S}^p$ for all $p>\frac{1}{2}$.
\end{corollary}

\begin{proof}
From Proposition 3.3, under the layer decomposition we have
\[
U_{s'} [S_{z_i}^*, S_{z_i}] U_{s'}^*
=
T_i T_i^* \otimes (I - B_d B_d^*) - T_i^* T_i \otimes (I - B_d^* B_d).
\]

Lemma 2.4 gives
\[
(I - B_d B_d^*) e_l^{(d)}
=
\frac{d-1}{l+d-1} e_l^{(d)},
\qquad
(I - B_d^* B_d) e_l^{(d)}
=
\frac{d-1}{l+d} e_l^{(d)}.
\]

Thus, on the \(l\)-th fiber
\[
\mathcal F_l = \operatorname{span}\{\theta_{s_i}^{(i)} \otimes e_l^{(d)} : s_i\},
\]
the restriction is
\[
R_l
:=
[S_{z_i}^*, S_{z_i}] \big|_{\mathcal F_l}
=
\frac{d-1}{l+d-1} T_i T_i^* \otimes I
-
\frac{d-1}{l+d} T_i^* T_i \otimes I.
\]

Since both \(T_i T_i^*\) and \(T_i^* T_i\) are rank-one operators on \(K_i\), their linear combination \(R_l\) has rank at most two. Let \(\sigma_{l,1} \ge \sigma_{l,2} \ge 0\) be its nonzero singular values. From the Hilbert-Schmidt norm computation,
\[
\|R_l\|_{\mathrm{HS}}^2
=
\frac{(d-1)^2}{(l+d-1)^2}
+
\frac{(d-1)^2}{(l+d)^2}
-
\frac{2(d-1)^2|\eta_i'(0)|^2}{(l+d-1)(l+d)}
= O(l^{-2}).
\]
Hence
\[
\sigma_{l,1}^2 + \sigma_{l,2}^2 = O(l^{-2}),
\]
and therefore
\[
\sigma_{l,1} = O(l^{-1}), \qquad \sigma_{l,2} = O(l^{-1}).
\]

Let \((\sigma_k)_{k\ge1}\) be the global sequence of singular values. The number of layers is \(D_i = \prod_{j\neq i} d_j\), and each fiber \(\mathcal F_l\) contributes at most two singular values. Consequently, for \(l=0,1,\dots,L\), there are \(D_i(L+1)\) fibers, contributing at most \(2D_i(L+1)\) singular values. Hence for \(K \approx 2D_iL\),
\[
\sigma_K \approx \frac{C}{L} \approx \frac{C'}{K}.
\]
Thus the global singular values satisfy
\[
\sigma_k = O(k^{-1}).
\]

It follows that there exists a constant \(C>0\) such that for all \(k\ge1\),
\[
\sigma_k \le C k^{-1}.
\]
For any \(p > 1\),
\[
\sum_{k=1}^{\infty} \sigma_k^p
\le
C^p \sum_{k=1}^{\infty} k^{-p}
< \infty.
\]
Therefore
\[
[S_{z_i}^*, S_{z_i}] \in \mathcal S^p \quad (p > 1).
\]

For \(p = 1\), from \(\sigma_{l,1} \ge c/(l+1)\) for some \(c>0\) (which follows from the explicit formula for \(\|R_l\|_{\mathrm{HS}}^2\)), we obtain
\[
\sum_{k=1}^{\infty} \sigma_k
\ge
\sum_{l=0}^{\infty} \sigma_{l,1}
\ge
\sum_{l=0}^{\infty} \frac{c}{l+1}
= \infty.
\]
Thus
\[
[S_{z_i}^*, S_{z_i}] \notin \mathcal S^1.
\]

From Proposition 4.1, under the bi-layer decomposition we have
\[
V_{s''} [S_{z_i}^*, S_{z_j}] V_{s''}^*
=
T_i^* \otimes T_j \otimes (B_d^* B_d - B_d B_d^*).
\]

Lemma 2.4 gives
\[
(B_d^* B_d - B_d B_d^*) e_l^{(d)}
=
\frac{d-1}{(l+d)(l+d-1)} e_l^{(d)}.
\]

Thus, on the \(l\)-th fiber
\[
\mathcal F_l = \operatorname{span}\{\theta_{s_i}^{(i)} \otimes \theta_{s_j}^{(j)} \otimes e_l^{(d)} : s_i, s_j\},
\]
the restriction is
\[
\widetilde R_l
:=
[S_{z_i}^*, S_{z_j}] \big|_{\mathcal F_l}
=
\frac{d-1}{(l+d)(l+d-1)} \, T_i^* \otimes T_j \otimes I.
\]

Since both \(T_i^*\) and \(T_j\) are rank-one operators, their tensor product \(T_i^* \otimes T_j\) is also rank-one, hence \(\operatorname{rank}(\widetilde R_l)=1\). Let \(\sigma_l\) be its unique nonzero singular value. Because
\[
\|T_i^* \otimes T_j\|_{\mathrm{HS}} = \|T_i^*\|_{\mathrm{HS}} \|T_j\|_{\mathrm{HS}} = 1,
\]
we obtain
\[
\sigma_l = \frac{d-1}{(l+d)(l+d-1)}.
\]
Consequently, as \(l\to\infty\),
\[
\sigma_l = O(l^{-2}).
\]

Let \((\sigma_k)_{k\ge1}\) be the global sequence of singular values. The number of bi-layers is \(D_{ij} = \prod_{k\neq i,j} d_k\), and each fiber \(\mathcal F_l\) contributes exactly one singular value \(\sigma_l\). Therefore, for \(l=0,1,\dots,L\), there are \(D_{ij}(L+1)\) fibers, contributing \(D_{ij}(L+1)\) singular values. Hence for \(K\approx D_{ij}L\),
\[
\sigma_K \approx \frac{C}{L^2} \approx \frac{C'}{K^2}.
\]
Thus the global singular values satisfy
\[
\sigma_k = O(k^{-2}).
\]

It follows that there exists a constant \(C>0\) such that for all \(k\ge1\),
\[
\sigma_k \le C k^{-2}.
\]
Therefore, for any \(p>\frac{1}{2}\),
\[
\sum_{k=1}^{\infty} \sigma_k^p
\le
C^p \sum_{k=1}^{\infty} k^{-2p}
< \infty.
\]
Hence
\[
[S_{z_i}^*, S_{z_j}] \in \mathcal S^p \quad (p>\frac{1}{2}).
\]

In particular, taking \(p=1\) gives \([S_{z_i}^*, S_{z_j}]\in \mathcal S^1\), i.e., the cross-commutator is trace class.
\end{proof}

\section*{Acknowledgment} The second author was supported partly by Chongqing Science and Technology Commission (Grant No.CSTB2025NSCQ-LZX0075), Science and Technology Innovation Key R\&D Program of Chongqing and Scientific and Technological Research Program of Chongqing Municipal Education Commission (Grant No.KJZD-K202301109).

\end{document}